\documentclass[reqno]{amsart}
\usepackage{amsmath,amsthm,amscd,amssymb,amsfonts, amsbsy}
\usepackage{latexsym, color, enumerate}
\usepackage{pxfonts}

\usepackage[hypertex]{hyperref}

\theoremstyle{plain}
\newtheorem{theorem}{Theorem}[section]
\newtheorem{lemma}[theorem]{Lemma}
\newtheorem{corollary}[theorem]{Corollary}

\theoremstyle{definition}
\newtheorem{definition}[theorem]{Definition}

\newtheorem{assumption}[theorem]{Assumption}

\theoremstyle{remark}
\newtheorem{remark}[theorem]{Remark}

\numberwithin{equation}{section}

\newcommand{\bR}{\mathbb{R}}

\providecommand{\set}[1]{\{#1\}}

\providecommand{\abs}[1]{\lvert#1\rvert}

\providecommand{\norm}[1]{\lVert#1\rVert}

\def\Xint#1{\mathchoice
	{\XXint\displaystyle\textstyle{#1}}%
	{\XXint\textstyle\scriptstyle{#1}}%
	{\XXint\scriptstyle\scriptscriptstyle{#1}}%
	{\XXint\scriptscriptstyle\scriptscriptstyle{#1}}%
	\!\int}
\def\XXint#1#2#3{{\setbox0=\hbox{$#1{#2#3}{\int}$}
		\vcenter{\hbox{$#2#3$}}\kern-.5\wd0}}
\def\dashint{\Xint-}

\newcommand{\p}{\partial}
\newcommand{\epsi}{\varepsilon}

\begin{document}

%
%
\subjclass[2010]{Primary 35J25, 35B65; Secondary 35J15}
\keywords{Oblique derivative problem, Lipschitz domains, $W^2_p$ estimate and solvability, fully nonlinear elliptic equations.}

	\title[oblique derivative problem]{On the $W^2_p$ estimate for oblique derivative problem in Lipschitz domains}
	
	\author[Hongjie Dong]{Hongjie Dong}	
	
	\address{Brown University, Division of Applied Mathematics,
		Providence RI 02906}
	
	\email{hongjie$\_$dong@brown.edu}
\thanks{H. Dong and Z. Li were partially supported by the NSF under agreement DMS-1600593.}
	
	\author[Zongyuan Li]{Zongyuan Li}
	
	\address{Brown University, Division of Applied Mathematics,
		Providence RI 02906}
	
	\email{zongyuan$\_$li@brown.edu}
	
\begin{abstract}
The aim of this paper is to establish $W^2_p$ estimate for non-divergence form second-order elliptic equations with the oblique derivative boundary condition in domains with small Lipschitz constants. Our result generalizes those in \cite{LiebermanBookParabolic, LiebermanBook}, which work for $C^{1,\alpha}$ domains with $\alpha > 1-1/p$. As an application, we also obtain a solvability result. An extension to fully nonlinear elliptic equations with the oblique derivative boundary condition is also discussed.
\end{abstract}
\today
\maketitle
\section{Introduction}
In this paper, we consider $W^2_p$ estimate for oblique derivative problem. For the problem \begin{equation*}
\begin{cases}
Lu :=a_{ij}D_{ij}u + a_iD_i u + a_0 u = f \quad &\mbox{in } \Omega,\\
Bu :=b_0 u + b_iD_i u = g \quad &\mbox{on } \partial \Omega,
\end{cases}
\end{equation*}
we aim to prove
$$\|u\|_{W^2_p(\Omega)} \leq N(\|f\|_{L_p(\Omega)}+\|g\|_{W^{1-1/p}_p(\p\Omega)} + \|u\|_{L_p(\Omega)}).$$
Here $L$ is uniformly elliptic, and oblique derivative means for some $\delta \in(0,1]$ $b \cdot n \geq \delta|b|$ almost everywhere. Here $n$ is the outer normal direction, which is defined almost everywhere on $\p\Omega$.

As the $W^2_p$ estimate for elliptic equations with uniformly continuous coefficients in smooth (say $C^{1,1}$) domains has been well studied for a long time (see e.g., \cite[Theorem~9.13]{GilbargTrudinger}), people are more interested in the case of discontinuous coefficients or rough domains.

Concerning discontinuous coefficients, the case when $a_{ij}$ belongs to the class of vanishing mean oscillation (VMO) is of particular interest. We say a function $f \in \text{VMO}$ if
\begin{equation*}
\omega(\rho):= \sup_{x, 0<r<\rho}\dashint_{B_r(x)}|f(y)-(f)_{B_r(x)}|\,dy \rightarrow 0 \text{ as } \rho \rightarrow 0,
\end{equation*}
where $(f)_{B_r(x)} := \frac{1}{|B_r|}\int_{B_r(x)}f(y)\,dy$ is the average of $f$ in $B_r(x)$.
		
The $W^2_p$ estimate for equations with VMO coefficients was first established by Chiarenza, Frasca, and Longo in \cite{VMOChiarenzaFrascaLongoInterior}. The method is mainly based on the representation formula: the Calder\'on-Zygmund theorem together with a commutator estimate. Later in \cite{KrylovSharp}, based on the so-called ``sharp function'' estimate and the Fefferman-Stein theorem, Krylov gave a unified proof of the $W^2_p$ estimate for parabolic/elliptic equations with VMO-in-$x$ coefficients. Furthermore, using this method, it is possible to relax the regularity assumptions on $a_{ij}$. Here we mention the following ``partially VMO'' condition for elliptic equations, which is extremely useful in discussing boundary value problems. Such ``partially VMO'' can be written as:
$$\omega'(\rho):=\sup_{x,0<r<\rho} \dashint_{x^d-r}^{x^d+r}\dashint_{B'_r(x')}|f(y',y^d)-(f)_{B'_r(x')}(y^d)|\,dy'dy^d \rightarrow 0, \text{ as } \rho \rightarrow 0,$$
where $x=(x',x^d)$ and $B'_r(x')$ is $(d-1)$-dimension ball. Such theory was developed in \cite{KimKrylov} by Kim and Krylov for $p>2$, and in \cite{DongVMOallp} by the first author here for all $p\in(1,\infty)$. It is worth noting that in \cite{DongVMOallp}, a more general regularity assumption called ``hierarchically partially VMO'' was discussed.

Such results allow us to consider Dirichlet or Neumann problems with VMO coefficients in the half space $\bR^d_{+}$: by simple extension and reflection we get equations in $\bR^d$ with ``partially-VMO'' coefficients. Based on this, in \cite{KimKrylov} a $W^2_p$ estimate for the oblique derivative problem in $\bR^d_{+}$ is also discussed via a perturbation argument. It turns out that the perturbation argument requires $b \in C^{\alpha}, \alpha > 1-1/p$. For details and history of discussing equations with ``partially VMO'' coefficients, one may refer to \cite{DongBookChapter,Krylovbook}.

In this paper, we will focus on general bounded domain $\Omega$ and its regularity assumptions regarding $W^2_p$ estimate. First noting that, by flattening the boundary, the aforementioned $W^2_p$ estimate in the half space will simply lead to corresponding results in $C^{1,1}$ domains. This is because a $C^{1,1}$ change of variables will preserve all the regularity assumptions on the elliptic operator $L$ and the boundary condition. Also, $W^2_p$ norms under these two coordinates are comparable.

For the oblique derivative problem, however, the smoothness assumption for $\Omega$ can be relaxed. The idea is to consider an extension problem in curved domain, which will reduce the boundary condition to be homogeneous. This will compensate the lack of regularity in our change of variables when flattening the boundary. In this way, Lieberman reduced the assumption to $\Omega \in C^{1,\alpha}, \alpha > 1-1/p$ in \cite{LiebermanBookParabolic, LiebermanBook}. In our paper, with the help of Hardy's inequality, we employ a new idea of extension. Together with a perturbation argument, we get the $W^2_p$ estimate in any small Lipschitz domain, i.e., domain with local representation function having sufficiently small Lipschitz constant.

We would like to mention that, there is also a ``Schauder type'' $C^{2,\alpha}$ estimate for oblique derivative problems in $C^{1,\alpha}$ domains. Such result was obtained by Lieberman in \cite{Lieberman}. One could notice that our result is in the same spirit: the regularity assumption on $\p\Omega$ is one derivative less than the corresponding Dirichlet problem. Later in \cite{Safonov}, Safonov came up with an alternative proof for this problem. His proof also includes an extension problem, which actually motivates us of this paper. It is worth noting that our perturbation argument in proving Theorem \ref{maintheorem} can replace Theorem 2.1 in \cite{Safonov}, which is to find a $C^{2,\alpha}$ diffeomorphism mapping $b\cdot D$ to $\frac{\p}{\p y^d}$. This can also be used to simplify the proof in \cite{Safonov}.

The structure of the paper is as follows. In Section \ref{notationsection} we first introduce the basic setups and notation, and then present our main result of the $W^2_p$ estimate in Theorem \ref{maintheorem}. We state the corresponding result regarding the existence and uniqueness of $W^2_p$ solutions in Theorem \ref{well-posednessThm}. Next in Section \ref{regularizedsection} we introduce the cylindrical neighborhood and a special choice of orthonormal systems which we are going to work with for the boundary estimate. Then as a preparation, we introduce the regularized distance which is a useful tool for rough boundaries. With all these, the proof of our main result Theorem \ref{maintheorem} is given in Section \ref{mainstepssection}. The solvability result Theorem \ref{well-posednessThm} is proved in Section \ref{solvabilitysection}. One advantage of our proof is that it also works for nonlinear equations with proper convexity conditions. We will prove one such result for Bellman equations in Section \ref{nonlinearsection}.

\section{Notation and Main Result}\label{notationsection}

Let $d>1$ be a positive integer. In this paper, we denote a point $x$ in $\mathbb{R}^d$ by $x=(x^1,x^2,\ldots,x^d)=(x',x^d)$. Also, we denote $$\mathbb{R}^d_{+} = \{x^d >0\},\quad  B_{R}^{+}=B_R\cap \mathbb{R}^d_{+}, \quad \text{and } B'_{R}=\{x'\in\mathbb{R}^{d-1} : |x'|<R\}.$$
Besides the usual partial derivative symbol $\frac{\p}{\p x^i}$, we use the following notation:
\begin{equation*}
D_i u = u_{x^i},\quad D_{ij}u = u_{x^i x^j}.
\end{equation*}
We write $W^k_p$ for Sobolev spaces, i.e., functions themselves and all derivatives up to order $k$ lie in $L_p$, and $\mathring{W}^1_p(\Omega)=\overline{C^\infty_c(\Omega)}$ under the $W^1_p$ norm. For fractional Sobolev spaces $W^s_p, s\in(0,1)$, the interpolation definition is used. Notice that $W^{1-1/p}_p(\p\Omega)$ is exactly the trace space of $W^1_p(\Omega)$ for $C^1$ (or small Lipschitz) domain. For simplicity, in this paper sometimes we write
$$
\|\cdot\|_p=\|\cdot\|_{L_p},\quad
\|\cdot\|_{p,\Omega}=\|\cdot\|_{L_p(\Omega)},\quad
\|\cdot\|_{1,p}=\|\cdot\|_{W^1_p},\quad\text{etc.}
$$
The summation convention, for instance,
$$
b\cdot D = b_iD_i : = \sum_{i=1}^{d}b_iD_i,
$$
is adopted throughout this paper.

For a Lipschitz continuous function $f$, denote $$[f]_1:=\sup_{x\neq y}\frac{|f(x)-f(y)|}{|x-y|}$$ for its Lipschitz constant. We write $C^{k,1}$ for the class of $k$-th order continuously differentiable functions with all $k$-th order derivatives being Lipschitz continuous.

In this paper, we will use the following notation for the average: $$(f)_{B_r(x)}=\dashint_{B_r(x)}f(y)\,dy:=\frac{1}{|{B_r(x)}|}\int_{_{B_r(x)}}f(y)\,dy.$$
	
When proving inequalities, we will use $N$ for the absolute constant (to be more specific, independent of the local radius parameter $r$ or $R$ in this paper). In the middle steps, we will omit the dependence of $N$ on $\Omega, a_{ij}, b_{i}$, etc. Also, $N$ may vary from line to line.

The main purpose of this paper is to derive the $W^2_p$ estimate in small Lipschitz domains. Let us first give the formal definition of an $\epsi_0$-Lipschitz domain.
\begin{definition}[$\epsi_0$-Lipschitz domain]\label{Defsmalllipschitzdomain}
A bounded domain $\Omega \subset \mathbb{R}^d$ is said to be $\epsi_0$-Lipschitz if for any $x_0 \in \p \Omega$, there exist $R_0>0$ independent of $x_0$, an orthonormal coordinate system $x=(x',x^d)$ centered at $x_0$, and a Lipschitz function $\psi_0$ such that
$$\Omega \cap B_{R_0}(x_0) = \{x \in B_{R_0}(0) : x^d < \psi_0(x')\},$$
\begin{equation}\label{epsilonlipschitz1}
|\psi_0(x'_1)-\psi_0(x'_2)|<\epsi_0|x'_1-x'_2|.
\end{equation}
Since Lipschitz function is almost everywhere differentiable, \eqref{epsilonlipschitz1} is equivalent to
\begin{equation}\label{epsilonlipschitz2}
|D\psi_0(x')|<\epsi_0 \quad \mbox{a.e. }
\end{equation}
\end{definition}
Notice that this definition is given in a natural choice of coordinate system, where the $x^d$-axis is chosen to be close to the normal direction of $\p\Omega$ at $x_0$.  In the next section, we will give a coordinate system adapted to the oblique derivative boundary condition, which is more convenient to work with.
	
The next part gives basic conditions which will be assumed throughout this paper. For a second-order elliptic equation
$$Lu:=a_{ij}D_{ij}u + a_{i}D_i u + a_0 u = f,$$
we assume
\begin{align}
&a_{ij} \in L_{\infty},\quad a_{ij}=a_{ji},\quad \nu |\xi|^2 \leq a_{ij}\xi_i\xi_j \leq \nu^{-1}|\xi|^2,\label{unifellipticciefficients}\\
&a_i, a_0 \in L_{\infty}(\Omega),\label{boundeda}
\end{align}
where $\nu\in (0,1]$ is a constant.

The following small BMO assumption will be assumed for $a_{ij}$, where $\theta$ is a constant to be specified later.
\begin{assumption}[$r_0,\theta$]\label{smallBMOAss}
For a constant $\theta>0$, there exists some $r_0>0$, such that $$\sup_{x, 0<r<r_0}\dashint_{B_r(x)\cap \Omega}|a_{ij}(y)-(a_{ij})_{B_r(x)\cap\Omega}|\,dy < \theta.$$
\end{assumption}
In this paper, we consider the following oblique derivative boundary condition:
\begin{equation}\label{obliqueboundarycondition}
Bu := b_0 u + b_iD_i u = g\quad \mbox{on } \partial \Omega.
\end{equation}
As usual, this is understood in the sense of trace.
For the coefficients $b_0$ and $b_i$, we assume
\begin{align}
&b_0, b_i \in C^\alpha(\partial \Omega),\quad  \mbox{where } \alpha > 1-\frac{1}{p},\label{bregularity}\\
&b\cdot n \geq \delta|b|\quad \text{a.e.} \quad\text{for some }\delta\in (0,1]. \label{oblique}
\end{align}
Here \eqref{oblique} represents the obliqueness of the vector field $b=(b_i)_{i=1}^d$ and $n$ is the unit outer normal direction. If $\p\Omega$ locally is represented by a Lipschitz function $y^d=\psi(y')$, we can write $n$ in terms of $D\psi$:
\begin{equation}\label{normal}
n=\frac{(D\psi,-1)}{\sqrt{|D\psi|^2+1}}\quad \text{a.e.}
\end{equation}
Now we state our main result.
\begin{theorem}[$W^2_p$ estimate in small Lipschitz domain]\label{maintheorem}
In a Lipschitz domain $\Omega$ (bounded or unbounded), suppose $u \in W^2_p(\Omega)$ solves		
\begin{equation}\label{eq11.52}		
\begin{cases}
Lu = f \quad &\mbox{in } \Omega,\\
Bu = g\quad &\mbox{on } \partial \Omega.
\end{cases}
\end{equation}
Assume assumptions \eqref{unifellipticciefficients}, \eqref{boundeda}, \eqref{bregularity}, and \eqref{oblique} hold, $f\in L_p(\Omega)$, and $g\in W^{1-1/p}_p(\p \Omega)$. There exist $\theta_0=\theta_0(d,p,\nu)>0$ and $\epsi_0=\epsi_0(d,p,\nu, \|a_{i},a_0\|_\infty)>0$ small enough, such that if Assumption \ref{smallBMOAss}$(r_0, \delta\theta_0)$  is satisfied, and the domain $\Omega$ is $\delta^2\epsi_0$-Lipschitz, then we have,
\begin{equation}\label{mainestimate}
\|u\|_{W^2_p(\Omega)} \leq N\big(\|u\|_{L_p(\Omega)} + \|f\|_{L_p(\Omega)} + \|g\|_{W^{1-1/p}_p(\partial\Omega)}\big).
\end{equation}
Here $N=N(d,p,\nu,\|a_{i},a_0\|_\infty,\|b_i,b_0\|_\alpha,r_0,R_0,\delta)$ is a constant.
In particular, the result holds when $\Omega \in C^1$ and $a_{ij}\in \text{VMO}$.
\end{theorem}

As an application we have the corresponding solvability result. For this, we also need the following conditions:
\begin{align}
&a_0\leq 0, \quad b_0 \ge 0\quad \text{a.e.} \label{sign}\\
&a_0\not\equiv0 \mbox{ or } b_0\not\equiv0,\label{notalwayszero}\\
&\alpha>\max\{1-1/d,1-1/p\},\label{alpha>d}
\end{align}	
where $\alpha$ is the H\"{o}lder exponent of $b$.
	
\begin{theorem}\label{well-posednessThm}
Besides assumptions in Theorem \ref{maintheorem} regarding $\Omega$, $L$, and $B$, we also assume conditions \eqref{sign}-\eqref{alpha>d} to hold and $\Omega$ to be bounded. Then there exist $\theta_0=\theta_0(d,p,\nu)>0$ and $\epsi_0=\epsi_0(d,p,\nu, \|a_{i},a_0\|_\infty)>0$ small enough, such that, if Assumption \ref{smallBMOAss} $(r_0, \delta\theta_0)$ is satisfied, and the domain $\Omega$ is $\delta^2\epsi_0$-Lipschitz, we have that for any $f\in L_p(\Omega), g\in W^{1-1/p}_p(\Omega)$, there exists a unique solution $u\in W^2_p(\Omega)$ of \eqref{eq11.52} with
\begin{equation}\label{nodependenceapriori}
\|u\|_{W^2_p(\Omega)}\leq N(\|f\|_{L_p(\Omega)}+\|g\|_{W^{1-1/p}_p(\partial \Omega)}).
\end{equation}
Here $N$ is a constant independent of $u$.
\end{theorem}

\begin{remark}\label{bregularityexample}
Our perturbation argument still works if the regularity assumption \eqref{bregularity} is replaced by the following: for $1\leq i \leq d$,
\begin{equation}
b_i \in
\begin{cases}\label{bSobolev}
W^1_p(\Omega), \text{ if } p>d,\\
W^1_q(\Omega)\text{ for some } q>d, \text{ if } p=d,\\
W^1_d(\Omega), \text{ if } p<d,
\end{cases}
\end{equation}
and
\begin{equation}\label{bSobolev2}
b_0 \in
\begin{cases}
W^1_p(\Omega), \text{ if } p>d/2,\\
W^1_q(\Omega)\text{ for some } q>d/2, \text{ if } p=d/2,\\
W^1_{d/2}(\Omega), \text{ if } p<d/2.
\end{cases}
\end{equation}
Here $d \geq 2$ is the space dimension.

The following example in $\mathbb{R}^2$ shows that the regularity assumption \eqref{bSobolev}-\eqref{bSobolev2} is sharp. Consider $\Omega=\{(x,y): x>|y|^{1+\epsi}\}$, $u=(x|y|^\beta+y)\eta_R$, where $\beta \in (0,1), \epsi>0$ are some constants to be determined later, $\eta$ is some smooth cutoff function supported in a ball $B_R$, and equals to $1$ in $B_{R/2}$. Direct calculation shows that if we choose $$\max\{1/2-(2+\epsi)/p, 1-\epsi + (2+\epsi)/p\}<\beta \leq 1-(2+\epsi)/p,$$
and
$$
b=(-1,|y|^\beta),
$$
then,
$$
u\in L_p(\Omega),\quad \Delta u \in L_p(\Omega),\quad  Bu=b\cdot Du \in W^1_p(\Omega),\quad  D_{12}u \notin L_p(\Omega).
$$
Hence \eqref{mainestimate} cannot be true in this case. Notice that \eqref{bSobolev} is violated since we only have $b\in W^1_q$ for $q<2/(1-\beta)< p$.
\end{remark}
\begin{remark}
To see the importance of the small Lipschitz condition, we give the following example which is also in $\mathbb{R}^2$. We use the polar coordinates $(r,\theta)$.	Let $\theta_0 \in (\pi/2,\pi)$ be a fixed angle. Consider the wedge domain
$$
\Omega^{\theta_0}=\{(r,\theta): -\theta_0<\theta<\theta_0\}\quad \text{and}\quad \Gamma^{\theta_0}=\{\theta=-\theta_0, \theta_0\}.$$
Define $u={\rm Im}\,\{e^{i(\alpha_0-1)\theta_0}z^{\alpha_0}\}\eta_R$, where $z=re^{i\theta}$, $\alpha_0$ is some constant to be determined later, and $\eta_R$ is the cutoff function in Remark \ref{bregularityexample}. Noting that the opposite of the $x$-direction is oblique on $\Gamma^{\theta_0}$, direct computation shows that on $\Gamma^{\theta_0} \cap B_{R/2}$,
$$\frac{\p u}{\p x} = \alpha_0 r^{\alpha_0-1}\sin\big((\alpha_0-1)(\theta_0+\theta)\big).$$
Hence, if we choose $b=(-1,0)$ and $\alpha_0=\pi/(2\theta_0)+1 \in (3/2,2)$, we have $Bu = 0$ on $\Gamma^{\theta_0} \cap B_{R/2}$. Now $u \in L_p$ for all $p$, $\Delta u = 0$, but for $p\geq \frac{2}{2-\alpha_0}$,
$$r^{\alpha_0-2} \lesssim|D^2u| \notin L_p.$$
\end{remark}

\section{Cylindrical Neighborhood and Regularized Distance}\label{regularizedsection}
In this paper, local properties near the boundary will be intensively studied. Rather than the coordinate system coming with Definition \ref{Defsmalllipschitzdomain}, it is more convenient to use the following coordinates $y=(y',y^d)$ which depend on the boundary condition \eqref{obliqueboundarycondition}. Also, it is more convenient to work with the following ``cylindrical'' neighborhood rather than the ``half ball'' neighborhood in Definition \ref{Defsmalllipschitzdomain}.

Consider an $\epsi_0$-Lipschitz domain $\Omega$, $x_0 \in \partial \Omega$ where $D\psi_0$ exists, and a vector field $b$ which satisfies \eqref{oblique} at $x_0$. We first take a rotation of the coordinates in Definition \ref{Defsmalllipschitzdomain} to make $y^d$-axis lie in $b(x_0)$ direction. By \eqref{oblique}, $b(x_0)$ is non-tangential at the point $x_0$. Taking \eqref{epsilonlipschitz2} into account, locally $\p\Omega$ is still a graph :$y^d = \psi(y')$. Here $\psi$ can be obtained from $\psi_0$ by the implicit function theorem. The small Lipschitz condition \eqref{epsilonlipschitz2} now can be written in terms of $(y',y^d)$ and $\psi$ as the smallness of the oscillation of $D\psi$. Notice that due to the rotation we will introduce a constant factor $1/\delta^2$ in front. To be specific, in Theorem \ref{maintheorem} we assume $\Omega$ to be $\delta^2\epsi_0$-Lipschitz, i.e. $$|D\psi_0(x')|<\delta^2\epsi_0 \quad \mbox{a.e. } x' \in B'_{R_0}.$$
Straightforward computation gives us, if $\epsi_0<1/8$, in the new coordinates
\begin{equation}\label{epsilonlipschitz}
|D\psi(y')-D\psi(z')|<3\epsi_0\quad \mbox{a.e. } y', z' \in B'_{\delta R_0}.
\end{equation}
Now due to the expression of $n$ in terms of $D\psi$ \eqref{normal}, the obliqueness condition \eqref{oblique} at the point $x_0=(0,\psi(0))$ can be written as:
$$|D\psi(0)|\leq \sqrt{1/\delta^2 -1}.$$
Choosing $\epsi_0<1/3$, we have
\begin{equation}\label{Dpsibounded}|D\psi|<2/\delta\quad \text{a.e. in } B'_{\delta R_0}.\end{equation}
We further shift the $y'$-coordinate plane so that $x_0=(0,\frac{3}{\delta}R)$. Here $R<\delta R_0$ is a radius parameter to be chosen small later. Due to \eqref{Dpsibounded}, we have
\begin{equation}\label{nottoosteep}
\forall y'\in B'_R(0), \quad  \frac{1}{\delta}R<\psi(y')< \frac{5}{\delta}R.
\end{equation}
The following is the ``cylindrical'' neighborhoods in which the local properties will be studied:
$$\Omega_R(x_0) := \{(y',y^d): y'\in B'_R(0), 0<y^d<\psi(y')\},$$
$$\Gamma_R(x_0) := \{(y',y^d): y'\in B'_R(0), y^d=\psi(y')\}.$$
$$Q_R(x_0) := \{(y',y^d): y'\in B'_R(0), 0<y^d<\frac{6}{\delta}R\},$$
The center $x_0$ will be omitted when there is no ambiguity.

The second part of this section is a useful tool for rough boundaries (say, worse than $C^2$). This is the regularized distance introduced by Lieberman in \cite{LiebermanRegularizedDistance}. For our problem, we modify Theorem 2.1 in \cite{LiebermanRegularizedDistance} to adapt to small Lipschitz domains.
\begin{theorem} [Local regularized distance for small Lipschitz domain]\label{RegularizedDistanceThm}
Let $\Omega$ be a bounded Lipschitz domain with local representation $\psi$ satisfying \eqref{epsilonlipschitz} and \eqref{Dpsibounded} on $\Omega_{4R}(x_0)$, $x_0 \in \partial \Omega$. Then there exists $\rho_0 \in C^\infty(Q_R\setminus \Gamma_R)\cap C^{0,1}(Q_R)$, such that:
\begin{align}
&\text{there exists a constant } M>0,\,\, \abs{D\rho_0}<M,\,\,  M^{-1}<\frac{\rho_0}{d_y}<M \mbox{ in } Q_R\setminus \Gamma_R,\label{distance}\\
&\mbox{where } d_y = \begin{cases}\text{dist}(y,\partial \Omega)\quad y\in\Omega_R\\-\text{dist}(y,\partial\Omega)\quad y\in Q_R\cap(\overline{\Omega_R})^c\end{cases},\nonumber\\
&|D\rho_0(y)-D\rho_0(z)|\leq 12\epsi_0 \quad\mbox{in } Q_R,\label{smalllip}\\
&|D^2\rho_0(x)|\leq N\frac{\epsi_0}{|\rho_0(x)|} \quad \mbox{in } Q_R\setminus\Gamma_R.\label{secondorderderivative}
\end{align}
Here $N>0$ is an absolute constant, and $M=M(\delta)$.
	
\end{theorem}
\begin{proof}
We follow the steps in \cite{LiebermanRegularizedDistance} for proving Theorems 1.1, 1.3, and 2.1. A sketch of the proof can be found in the appendix.
\end{proof}
In the above theorem, \eqref{distance} means $\rho_0$ is a local distance function. ``Regularized'' refers to the fact that this distance is $C^\infty$ in the interior. The expression \eqref{smalllip} is the small Lipschitz condition for $\rho_0$. The $\epsi_0$ in \eqref{secondorderderivative} is important in our proof.

Regularized distance can work as a suitable function to flatten the boundary, since it is smooth in the interior with a suitable growth rate of higher order derivatives near the boundary. Besides, one could also use it for mollification: $$\widetilde{g} (x) =: \int_{\Omega_R} g(x-\frac{\rho_0(x)}{M_1}y)\phi(y)\,dy.$$ The advantage is: besides $\|\widetilde{g}\|_{W^1_p} \lesssim \|g\|_{W^1_p}$, we have also nice control of $D^2\widetilde{g}$. Details will be given in Lemma \ref{YoungLem} and in the proof of Theorem \ref{ExtensionThm}.

\section{Proof of Theorem \ref{maintheorem}}\label{mainstepssection}
	
Now we are going to prove Theorem \ref{maintheorem}. Our proof is divided into three steps.  First we deal with the following model problem with a simple boundary condition.
\begin{lemma} \label{homogeneouslemma}
Consider a cylindrical neighborhood together with its top boundary $\Omega_R, \Gamma_R$, and the representation function $\psi$ as in Theorem \ref{RegularizedDistanceThm}. Here we take $R<\frac{\delta R_0}{4}$. Assume $u\in W^2_p(\Omega_R)$ solves $$\begin{cases}Lu = f \quad \mbox{in } \Omega_R,\\ \frac{\partial u}{\partial y^d} = 0 \quad \mbox{on } \Gamma_R. \end{cases}$$
There exist $\theta_0=\theta_0(d,p,\nu)>0$ and $\epsi_0=\epsi_0(d,p,\nu, \|a_{i},a_0\|_\infty)>0$ small enough, such that if Assumption \ref{smallBMOAss} holds with $(r_0, \delta\theta_0)$, and \eqref{epsilonlipschitz} holds with $3\epsi_0$, then for any $r<R$, we have
	
\begin{equation*}
\|D^2 u\|_{p,\Omega_{r/2}} \leq N(r^{-2}\|u\|_{p,\Omega_{r}} + \|f\|_{p,\Omega_{r}}).
\end{equation*}
Here $N=N(d,p,\nu,\|a_i,a_0\|_\infty, r_0,\delta)$ is a constant.
\end{lemma}

\begin{proof}
In $\Omega_R$, we flatten the boundary using the regularized distance in Theorem \ref{RegularizedDistanceThm}. In other words, we take the change of variables $z=\Phi(y)$: $z'=y'$, $z^d=\rho_0(y)$. This maps curved boundary $\Gamma_R$ to a flat portion of $\{z^d=0\}$.

Write $\widetilde{u}(z) = u(y(z))$. In the $z$ variables, the equation can be written as
\begin{equation*}
\begin{cases}
\widetilde{a_{ij}}D^z_{ij} \widetilde{u} + \widetilde{a_i}D^z_{i}\widetilde{u} + \widetilde{a_0} \widetilde{u} = \widetilde{f} \quad &\mbox{in } \Phi(\Omega_R)\subset \mathbb{R}^d_{+},\\
\frac{\partial \widetilde{u}}{\partial z^d} = 0 \quad &\mbox{on } \Phi(\Gamma_R) \subset \{z^d = 0\},
\end{cases}
\end{equation*}
where \begin{equation*}
\widetilde{f} = f - a_{ij}D_{ij}\rho_0 \frac{\partial \widetilde{u}}{\partial z^d},
\end{equation*}
and $\widetilde{a_i}(z) = a_{k}(y) \frac{\partial z^i}{\partial y^k},  \widetilde{a_0}(z)=a_0(y) \in L_\infty$. Noting $D\rho_0$ has small $L_\infty$-oscillation \eqref{smalllip}, we can choose $\epsi_0$ and $\theta_0$ small enough, to make $\widetilde{a_{ij}} = a_{kl} \frac{\partial z^i}{\partial y^k}\frac{\partial z^j}{\partial y^l}$ have as small BMO semi-norm as we want.

Now we apply the $W^2_p$ estimate for second-order elliptic equations with small BMO coefficient in half space and zero Neumann boundary condition. For this, first one could find in \cite{DongVMOallp} such result in $\mathbb{R}^d$. To deal with the Neumann boundary condition, we just take even extension for $u$ and $f$, and correspondingly for $\widetilde{a_{ij}}$ as in \cite{KimKrylov}. Noting that the extended equation has small partially BMO coefficients, this gives us the global $W^2_p$ estimate for small BMO coefficient in half space with zero Neumann data. One last thing to mention is that here the small BMO assumption is given in terms of the $z$ variables, when translating back to the $y$ variables we will have a $\delta$ factor in front due to the stretching in the change of variables that map balls to ellipses.

Localizing and using a dilation argument, we have for any $t/2\leq s<t\leq R$:
\begin{equation}\label{standNeumanFlat}
\|D^2_z\widetilde{u}\|_{p, s} \leq N \big((t-s)^{-2}\|\widetilde{u}\|_{p, t} + \|\widetilde{f}\|_{p, t}\big),
\end{equation}
where $N=N(d,p,\nu,\|a_{i},a_0\|_\infty,r_0)$. 	Here we used the abbreviation $\|\cdot\|_{p,s}:=\|\cdot\|_{L_p(\Phi(\Omega_s))}$.

We are left to estimate $\|\widetilde{f}\|_{L_p}$. For this, we use the property of regularized distance  \eqref{secondorderderivative}:
$$|D^2\rho_0(x)|\leq N\frac{\epsi_0}{|\rho_0(x)|}.$$
Combining this and Hardy's inequality, we obtain
\begin{equation}\label{hijestimate}
\big\|D_{ij}\rho_0 \frac{\partial \widetilde{u}}{\partial z^d}\big\|_p \leq N\epsi_0\big\|\frac{1}{\rho_0}\frac{\partial \widetilde{u}}{\partial z^d}\big\|_p \leq N\epsi_0\big\|\frac{1}{z^d}\frac{\partial \widetilde{u}}{\partial z^d}\big\|_p \leq N\epsi_0\|D^2_z \widetilde{u}\|_p.
\end{equation}
Here we used $z^d \lesssim  \text{dist}(y,\Gamma_R) \lesssim \rho_0$. Substituting into \eqref{standNeumanFlat}, we get
\begin{equation}\label{beforeiteration}
\|D^2_z\widetilde{u}\|_{p, s} \leq N \big((t-s)^{-2}\|\widetilde{u}\|_{p,t} + \|f\|_{p,t}\big) + N\epsi_0\|a_{ij}\|_\infty\|D^2_z \widetilde{u}\|_{p,t}.
\end{equation}
Choosing $\epsi_0$ small enough, such that $N\epsi_0\|a_{ij}\|_\infty<\frac{1}{5}$, we can use iteration argument to absorb $\|D^2_z\widetilde{u}\|_p$. Indeed, consider a sequence of balls $\{B_{r_k}: r_k = r - 2^{-k-1}r, k=0,1,\ldots\}$. Using \eqref{beforeiteration} with $s=r_k, t=r_{k+1}$ and summing in $k$, we have
\begin{equation*}
\sum_{k=0}^{\infty} 5^{-k}\|D^2_z\widetilde{u}\|_{p,r_k}\leq \sum_{k=0}^\infty \big(N\cdot 4^{k+2}\cdot 5^{-k} r^{-2}\|\widetilde{u}\|_{p,r} + 5^{-k-1}\|D^2_z\widetilde{u}\|_{p,r_{k+1}}\big).
\end{equation*}
This gives us:
\begin{equation*}
\|D^2_z\widetilde{u}\|_{p,\Phi(\Omega_{r/2})} \leq N (r^{-2}\|\widetilde{u}\|_{p,\Phi(\Omega_r)} + \|f\|_{p,\Phi(\Omega_r)}).
\end{equation*}
Now we get the desired estimate, but in the $z$ variables. If we change back to our original $y$ variables, we get similar singular terms, i.e., the term with $D^2\rho_0$:
\begin{equation*}
D_{ij}^{y}=\frac{\p z^k}{\p y^i}\frac{\p z^l}{\p y^j}D_{kl}^z + \frac{\p^2 z^k}{\p y^i \p y^j}D_k^z = \frac{\p z^k}{\p y^i}\frac{\p z^l}{\p y^j}D_{kl}^z + \frac{\p^2 \rho_0}{\p y^i \p y^j}D_d^z.
\end{equation*}
As in \eqref{hijestimate}, we can prove
$$
\|D^y_{ij}u\|_{L_p(\Omega_{r/2})} \leq N \|\widetilde{u}\|_{W^2_p(\Phi(\Omega_{r/2}))}.
$$
Notice that again the boundary condition $\frac{\partial \widetilde{u}}{\partial z^d} = 0$ is used when applying Hardy's inequality.
\end{proof}

The next part deals with an extension theorem. Our construction uses similar idea to \cite{Safonov}. Before we start, let us first formally introduce the mollification which we have mentioned in the previous section.

Consider $g \in W^{1-1/p}_p(\partial \Omega)$. First we extend $g$ to the interior in the usual way,
\begin{equation}\label{usualextension}
E: g\in W^{1-1/p}_p(\partial \Omega)\mapsto E(g)\in W^1_p(\Omega)
\end{equation}
with $\|E(g)\|_{W^1_p(\Omega)} \lesssim \|g\|_{W^{1-1/p}_p(\partial \Omega)}$. For simplicity, we will not distinguish $E(g)$ from $g$ in the following.

Now we can give the definition of our mollification.
\begin{definition}[Mollification using regularized distance] \label{DefMollification}
Suppose $\Omega$ is a bounded domain with small Lipschitz property \eqref{epsilonlipschitz}. If the regularized distance $\rho_0$ is defined on $\Omega_{2R}$, we can define in $\Omega_R$:
\begin{equation*}
\widetilde{g}(y) := \int g\Big(y-\frac{\rho_0(y)}{M_1}w\Big)\phi(w)\,dw.
\end{equation*}
Here, $$\phi \in C^\infty_c(B_1) \text{ with } \phi > 0, \int \phi = 1,$$
$$M_1 := \max\big\{\frac{3}{\delta}M,2\norm{D\rho_0}_{L_\infty}\big\}.$$
Recall in \eqref{distance}, $M$ is a constant such that $\rho_0(y) \leq Md_y$, where $d_y$ is the distance to the boundary.
\end{definition}

Clearly, $\widetilde{g}\in C^\infty(\Omega_R)$. We also have the following,
\begin{lemma}\label{YoungLem}
Let $g \in W^1_p(\Omega)$. Consider $\Omega, \Omega_R$ with $R<\frac{\delta R_0}{8}$ and $\widetilde{g}$ defined as above. Then we have,
\begin{align}
\|\widetilde{g}\|_{L_{p}(\Omega_R)}&\leq N\|g\|_{L_{p}(\Omega_{2R})}\label{young_Lp},\\
\|D\widetilde{g}\|_{L_p(\Omega_R)}&\leq N\|Dg\|_{L_p(\Omega_{2R})}\label{young_W1p}.
\end{align}
Here $N=N(p)$ is a constant.
\end{lemma}
\begin{proof}
See Appendix. Note that because of \eqref{distance}, our choice of $M_1$ guarantees $y-\frac{\rho_0(y)}{M_1}w \in \Omega_{2R}$ for any $y\in \Omega_R$ and $w \in B_1$.
\end{proof}
For the extension problem, we need the following inequality which is dual to Hardy's inequality.
\begin{lemma}\label{dualhardyLem}
For any $h \in L_p(0,1)$ where $p\in[1,\infty)$, we have $$\Big\|\int_{0}^{x}\frac{1}{1-t}h(t)\,dt\Big\|_{L_p(0,1)} \leq N(p)\|h\|_{L_p(0,1)}.$$
Again the constant $N$ only depends on $p$.
\end{lemma}
\begin{proof} We prove by a duality argument. For any $\|\eta\|_{L_{p'}(0,1)}=1$ where $p'$ satisfies $\frac{1}{p} + \frac{1}{p'}=1$, we have
\begin{align}
\int_{0}^{1}\eta(x)\int_{0}^{x}\frac{1}{1-t}h(t)\,dt\,dx &=\int_{0}^{1}h(t)\frac{1}{1-t}\Big(\int_{t}^{1}\eta(x)\,dx\Big)\,dt\label{changeorder}\\
&\leq \|h\|_{L_p(0,1)}\big\|\frac{1}{s}\int_0^s \eta(1-y)\,dy\big\|_{L_{p'}(0,1)}\label{holder}\\
&\leq N(p)\|h\|_{L_p(0,1)}\|\eta\|_{L_{p'}(0,1)}.\label{hardy}
\end{align}
Here we used Fubini's theorem in \eqref{changeorder}, H\"{o}lder's inequality in \eqref{holder}, and Hardy's inequality in \eqref{hardy} noting that $p'>1$.
\end{proof}

Now we can state our extension theorem.
\begin{theorem}\label{ExtensionThm}
Let $\Omega, \Omega_R$ defined as before, $R<\frac{\delta R_0}{8}$, $g \in W^{1-1/p}_p(\partial \Omega)$. Then we can find $v \in W^2_p(\Omega_R)$, such that
\begin{equation*}
\begin{cases}
\frac{\partial v}{\partial y^d} = g\quad \mbox{on } \Gamma_R,\\
\|v\|_{W^2_p(\Omega_R)} \leq N(\|g\|_{W^{1-1/p}_p(\Gamma_{3R})} + R^{-1+1/p}\|g\|_{L_p(\Gamma_{3R})}),
\end{cases}
\end{equation*}
where $N=N(\delta,p)$ is a constant.
\end{theorem}
\begin{proof}
Let $\widetilde{g}$ defined as in Definition \ref{DefMollification}. We define our extension as
\begin{equation*}
v= \int_{0}^{y^d} \widetilde{g}(y',t)\,dt \quad\mbox{for } y'\in B'_R,\, 0<y^d<\psi(y').
\end{equation*}
Since $\widetilde{g}\big|_{\Gamma_R} = g$, we have $\frac{\p v}{\p y^d} = g$ on $\Gamma_R$. We first estimate $\|v\|_{L_p}$:
\begin{equation}\label{vLp}
\|v\|_{L_p(\Omega_R)}\leq\Big\|\frac{\psi}{y^d}\int_{0}^{y^d}\widetilde{g}\Big\|_{L_{p}(\Omega_R)}
\lesssim \frac{R}{\delta}\|g\|_{L_{p}(\Omega_{2R})}.
\end{equation}
Here to get the last inequality, we used \eqref{nottoosteep}, Hardy's inequality, and \eqref{young_Lp}. Similarly,
\begin{align*}
\|Dv\|_{L_p(\Omega_R)}&\leq \|\widetilde{g}\|_{L_p(\Omega_R)} + \Big\|\int_{0}^{y^d}D_{y'}\widetilde{g}\Big\|_{L_p(\Omega_R)}\\
&\lesssim\|g\|_{L_p(\Omega_{2R})} + \frac{R}{\delta}\|D_{y'}g\|_{L_p(\Omega_{2R})}.
\end{align*}
Now we estimate $D^2v$. Noting that $\frac{\partial^2}{\partial y^i \partial y^d} v = \frac{\partial}{\partial y^i}\widetilde{g}$, we have $$\|DD_d v\|_{L_p(\Omega_R)} \leq \|D\widetilde{g}\|_{L_p(\Omega_R)}\lesssim \|Dg\|_{L_p(\Omega_{2R})}.$$
We are left to estimate $\frac{\p^2}{\p y^i\p y^j}v$ with $i,j<d$. By the chain rule,
\begin{align*}
&\frac{\p^2}{\p y^i\p y^j}v = \int_{0}^{y^d} \frac{\partial^2}{\partial y^i \partial y^j}\widetilde{g}(y',t)\,dt \\
&= \int_{0}^{y^d}\frac{\partial^2}{\partial y^i \partial y^j}\int g\big((y',t)-\frac{\rho_0(y',t)}{M_1}w\big)\phi(w)\,dw\,dt\\
&= \int_{0}^{y^d} \frac{\partial}{\partial y^i}\int \big[(D_j g)\big((y',t)-\frac{\rho_0}{M_1}w\big) +(-\frac{w^k}{M_1}D_j\rho_0)(D_kg)\big((y',t)-\frac{\rho_0}{M_1}w\big)\big]\phi(w)\,dw \,dt.
\end{align*}
Now we make a change of variables $w \mapsto z$:
$$z=(y',t)-\frac{\rho_0(y',t)}{M_1}w.$$
In the following, for simplicity we omit the dependence of $\rho_0$ on $(y',t)$ and $Dg$ on $z$ since there will be no ambiguity. Then we have
\begin{align*}
\frac{\p^2}{\p y^i\p y^j}v =& \int_{0}^{y^d}\frac{\p}{\p y^i}\int  \big(D_j g(z) - \frac{\big((y',t)-z\big)_k}{\rho_0}D_j\rho_0D_kg(z)\big)\phi
\big(\frac{(y',t)-z}{\rho_0}M_1\big)\big(\frac{M_1}{\rho_0}\big)^d\,dz \,dt\\
=& \int_0^{y^d} \int \big(\frac{-\delta_{ik}}{\rho_0}D_j\rho_0D_kg + \frac{w_k}{M_1} \frac{D_i\rho_0 D_j\rho_0}{\rho_0}D_kg - \frac{w_k}{M_1}D_{ij}\rho_0D_kg\big)\phi(w)\,dw \,dt\\
&+ \int_0^{y^d} \int\big(D_jg - \frac{w_k}{M_1}D_j\rho_0D_kg\big)\big(D_i\phi\frac{M_1}{\rho_0} - D_k\phi\cdot w_k \frac{D_i\rho_0}{\rho_0}\big)\,dw\,dt\\
&+ \int_0^{y^d} \int \big(D_jg - \frac{w_k}{M_1}D_j\rho_0D_kg\big)\phi(w)\big(-d\frac{D_i\rho_0}{\rho_0}\big)\,dw\,dt.
\end{align*}
	
From Theorem \ref{RegularizedDistanceThm} we know that $D\rho_0$ is bounded, and $\frac{\p^2 \rho_0}{\p y^i \p y^j} \lesssim \frac{1}{\rho_0}$ (no smallness is needed here). Also noting that $\phi$ has compact support in $B_1$, we have,
\begin{equation*}
\big|\int_{0}^{y^d} \frac{\partial^2}{\partial y^i \partial y^j}\widetilde{g}(y',t)\,dt\big| \leq N\big(\int_0^{y^d}\frac{1}{\rho_0}\int |Dg((y',t)-\frac{\rho_0(y',t)}{M_1}w)|(|\phi| + |D\phi|)\,dw\,dt\big).
\end{equation*}
Now, using properties of the regularized distance $\rho_0$, noting that $\rho_0(y',t), d_{(y',t)}$, and $\psi(y') - t$ all characterize the distance to the boundary, we have:
\begin{equation*}
\frac{1}{\rho_0(y',t)} \lesssim \frac{1}{\psi(y')-t}.
\end{equation*}
Then,
\begin{align}
&\big\|\int_{0}^{y^d} \frac{\partial^2}{\partial y^i \partial y^j}\widetilde{g}(y',t)\,dt\big\|^p_{L_p(\Omega_R)}\nonumber\\
&\leq N \int_{B'_R}\int_{0}^{\psi(y')} \bigg|\int_{0}^{y^d}\frac{1}{\psi(y')-t}\int_{B_1} \big|Dg\big((y',t)-\frac{\rho_0(y',t)}{M_1}w\big)\big|(|\phi|+|D\phi|)\,dw\,dt\bigg|^p \,dy^d\,dy'\nonumber\\
\label{nothardy}&\leq N \int_{B'_R}\left\|\int_{B_1} \big|Dg\big((y',\cdot)-\frac{\rho_0(y',\cdot)}{M_1}w\big)\big|(|\phi|+|D\phi|)\,dw\right\|^p_{L_p((0,\psi(y'))}\,dy'\\
\label{likeminkowski}&\leq N\int_{B'_{2R}}\|Dg(y',\cdot)\|^p_{L_p((0,\psi(y'))}\,dy'\\
&= N\|Dg\|_{L_p(\Omega_{2R})}\nonumber.
\end{align}
The inequality \eqref{nothardy} follows from Lemma \ref{dualhardyLem} and a dilation argument with the help of \eqref{nottoosteep}. We used the Minkowski inequality to prove \eqref{likeminkowski}, which is similar to the proof of \eqref{young_Lp} in Appendix.

Finally, to get the estimate with only local boundary norms as in our lemma, we use a localization argument. Consider $\eta \in C^\infty_c(Q_{3R})$ with $\eta = 1$ in $Q_{2R}$, $D\eta\lesssim 1/R$. We have
\begin{equation*}
\|E(\eta g)\|_{W^1_p(\Omega_{2R})} \leq N\|\eta g\|_{W^{1-1/p}_p(\p \Omega)}\leq N(\|g\|_{W^{1-1/p}_p(\Gamma_{3R})} + R^{-1+1/p}\|g\|_{L_p(\Gamma_{3R})}).
\end{equation*}
Here $E$ is the extension operator defined in \eqref{usualextension}. Replacing $g$ by $E(\eta g)$ in the proof above, we reach the desired inequality. The theorem is proved.
\end{proof}
Now, we have all the required ingredients for proving Theorem \ref{maintheorem}.
\begin{proof}[Proof of Theorem \ref{maintheorem}]
By interpolation, we only need to prove
$$\|D^2u\|_{L_p(\Omega)} \leq N(\|f\|_{L_p(\Omega)} + \|g\|_{W^{1-1/p}_p(\p\Omega)} + \|u\|_{W^1_p(\Omega)}).$$
We first give a boundary estimate in $\Omega_{\delta R_0/8}(x_0), x_0\in\p\Omega$. We apply Theorem \ref{ExtensionThm} with $R$ replaced by $r<\frac{\delta R_0}{8}$, and $g$ replaced by
$$
h:= g-\sum_{i=1}^d(b_i - b_i(x_0))D_i u - b_0 u.
$$
We find $v\in W^2_p(\Omega_r)$ such that $\frac{\p v}{\p y^d} = h$ on $\Gamma_r$ with the following:
\begin{align}
\|v\|_{W^2_p(\Omega_r)} \leq& N(\|h\|_{W^{1-1/p}_p(\Gamma_{3r})} + r^{-1+1/p}\|h\|_{L_p(\Gamma_{3r})})\nonumber\\
\leq&N\big((1+r^{-1+1/p})\|g\|_{W^{1-1/p}_p(\Gamma_{3r})} + (1+r^{-1+1/p})\|b_0 u\|_{W^{1-1/p}_p(\Gamma_{3r})}  \nonumber\\&+  \|(b_i - b_i(x_0))D_i u\|_{W^{1-1/p}_p(\Gamma_{3r})} + r^{-1+1/p} \|(b_i - b_i(x_0))D_i u\|_{L_p(\Gamma_{3r})}\big)\nonumber\\
\leq& N(r)\big(\|g\|_{W^{1-1/p}_p(\Gamma_{3r})}+ \|b_0\|_{C^\alpha(\Gamma_{3r})}
\|u\|_{W^{1-1/p}_p(\Gamma_{3r})}\big) + N\|b_i\|_{C^\alpha(\p\Omega)}\|Du\|_{L_p(\Gamma_{3r})} \label{HolderSobolev}\\
&+ \|b_i - b_i(x_0)\|_{L_{\infty}(\Gamma_{3r})}
\|Du\|_{W^{1-1/p}_p(\Gamma_{3r})}
+ N(r)\|b_i\|_{L_\infty(\Omega)}
\|Du\|_{L_p(\Gamma_{3r})}\big)\nonumber\\
\leq& N(r)\left(\|g\|_{W^{1-1/p}_p(\Gamma_{3r})}+\|u\|_{W^1_p(\Omega_{4r})}\right) + Nr^\alpha\|b_i\|_{C^\alpha}\|u\|_{W^2_p(\Omega_{4r})}\label{estimatev}.
\end{align}
Here we only write down the dependence $N=N(r)$ explicitly, and omit the dependence on $d,p,\nu,\norm{b_i}_{C^\alpha}$, etc. In \eqref{HolderSobolev} we used the inequality
$$\|fg\|_{W^{1-1/p}_p} \lesssim \|f\|_{C^\alpha}\|g\|_{L_p}+\|f\|_{L_\infty}\|g\|_{W^{1-1/p}_p},$$
provided that $\alpha > 1-1/p$. From \eqref{vLp} we also have the following estimate for lower order terms:
\begin{equation}\label{zerothorderv}
\|v\|_{L_p(\Omega_r)}\lesssim \frac{r}{\delta}\|h\|_{L_p(\Omega_{2r})} \lesssim r\|g\|_{L_p(\Omega_{2r})} + r^{1+\alpha}\|Du\|_{L_p(\Omega_{2r})} + r\|u\|_{L_p(\Omega_{2r})}.
\end{equation}
Now, $u-v$ solves
$$\begin{cases}
L(u-v) = f - Lv &\mbox{ in }\Omega_r\\
\frac{\partial (u-v)}{\partial y^d} = 0 &\mbox{ on }\Gamma_r.\end{cases}$$
We apply Lemma \ref{homogeneouslemma} to obtain  $$\|D^2 (u-v)\|_{p,\Omega_{r/2}} \leq N(r^{-2}\|u-v\|_{p,\Omega_{r}} + \|f-Lv\|_{p,\Omega_{r}}).
$$
Then,
\begin{align}
\|D^2 u\|_{p,\Omega_{r/2}} \leq& \|D^2 v\|_{p,\Omega_{r/2}} + N(r^{-2}\|u\|_{p,\Omega_{r}} + r^{-2}\|v\|_{p,\Omega_{r}} + \|f\|_{p,\Omega_{r}} + \|Lv\|_{p,\Omega_{r}})\nonumber\\
\leq& N(r)(\|u\|_{p,\Omega_{r}} + \|f\|_{p,\Omega_{r}} + \|v\|_{p,\Omega_r}) + N \|v\|_{W^2_ p(\Omega_{r})}\nonumber\\
\leq& N(r)\left(\|u\|_{W^1_p(\Omega_{4r})} + \|f\|_{L_p(\Omega_{r})} + \|g\|_{W^{1-1/p}_p(\Gamma_{3r})}\right) \label{applyestimatev}\\ &+ Nr^\alpha\|u\|_{W^2_p(\Omega_{4r})}\nonumber.
\end{align}
Here we applied \eqref{estimatev} and \eqref{zerothorderv} to get \eqref{applyestimatev}. This gives us a local boundary estimate. Combining this and the interior $W^2_p$ estimate in \cite{VMOChiarenzaFrascaLongoInterior, DongVMOallp}, we get the global estimate \eqref{mainestimate} using a standard partition of unity argument. For this, we cover $\bar{\Omega}$ with one interior portion and finitely many boundary "half balls" $\Omega_{r/2}$. Finally, we get
\begin{equation}\label{cover}
\|u\|_{W^2_p(\Omega)} \leq N_1(\|u\|_{L_p(\Omega)} + \|f\|_{L_p(\Omega)} + \|g\|_{W^{1-1/p}_p(\partial\Omega)}) + N_2N_3 r^\alpha \|u\|_{W^2_p(\Omega)}.
\end{equation}
Here $N_1,N_2$ are both constants depending on $d,p,\nu,\norm{a_i,a_0}_{\infty},\norm{b_i,b_0}_{C^\alpha},r_0,\delta$, and $N_1$ also depends on $r$. The constant $N_3$ depends on the ratio ``$\frac{\mbox{covering times by  }\Omega_{3r}}{\mbox{covering times by }\Omega_{r/2}}$'' which can be bounded regardless of $r$. Now we are left to choose $r$ small enough such that $r< \min\set{(\frac{1}{2MN_2})^{1/\alpha}, \frac{\delta R_0}{8}}$ to absorb the last term in \eqref{cover} into the left-hand side.
\end{proof}

\section{Application: Solvability}\label{solvabilitysection}

In this section, we give the proof of Theorem \ref{well-posednessThm}. For this we first remove $\|u\|_{L_p}$ from the right-hand side of \eqref{mainestimate} for the operator $L-\lambda$ with $\lambda$ large enough. We use a classical argument which can be found in \cite{GilbargTrudinger} and \cite{LiebermanBook}. As a result, in Corollary \ref{largelambdaexistenceCor} we get the $W^2_p$ well-posedness for large $\lambda$.
\begin{lemma}\label{removeuLem}
Under the assumptions of Theorem \ref{maintheorem}, we can find $\lambda_0$ depending on $d,p,\nu,\|a_{i},a_0\|_\infty,\|b_i,b_0\|_\alpha,r_0,R_0, \delta$ large enough, such that for any $\lambda\geq\lambda_0$ and $u \in W^2_p(\Omega)$ solving
\begin{equation}\label{largelambda}
\begin{cases}
(L-\lambda)u = f \quad &\mbox{in } \Omega,\\
Bu = g\quad &\mbox{on } \partial \Omega,
\end{cases}
\end{equation}
we have
\begin{equation}\label{removezeroorder}
\|u\|_{W^2_p(\Omega)} + \lambda \|u\|_{L_p(\Omega)}\leq N\big(\|f\|_{L_p(\Omega)} + \|g\|_{W^{1-1/p}_p(\partial\Omega)} + \sqrt{\lambda}\|g\|_{L_p(\Omega)}\big).
\end{equation}
Here $N=N(d,p,\nu,\|a_{i},a_0\|_\infty,\|b_i,b_0\|_\alpha,r_0,R_0,\delta)$ is a constant.
\end{lemma}
\begin{proof}
First, notice that when proving Theorem \ref{maintheorem}, actually we have proved a slightly stronger result:\\
Suppose $\Omega$ is a bounded domain with a ``small Lipschitz'' portion $T \subset \p \Omega$, and $\Omega'\subset \Omega$ with $\overline{\Omega'}\subset \Omega \cup T$. Then
\begin{equation} \label{partialboundary}
\|u\|_{W^2_p(\Omega')} \leq N\big(\|u\|_{L_p(\Omega)} + \|Lu\|_{L_p(\Omega)} + \|Bu\|_{W^{1-1/p}_p(T)}\big).
\end{equation}
Introduce a new space variable $x^{n+1}$, and let $v:=u(x)\cos(\sqrt{\lambda} x^{n+1})$. On $\Sigma := \Omega \times (-1,1), T:= \p \Omega\times (-1,1)$, $v$ satisfies
\begin{equation*}
\begin{cases}
\mathcal{L}v:= (L + D_{n+1,n+1})v = \cos(\sqrt\lambda x^{n+1})f(x) \quad &\mbox{in }\Sigma,\\
\mathcal{B}v = \cos(\sqrt{\lambda}x^{n+1})Bu = \cos(\sqrt{\lambda}x^{n+1}) g(x)\quad &\mbox{on } T.
\end{cases}
\end{equation*}
Applying \eqref{partialboundary} with $\Sigma, T$, and $\Sigma' := \Omega \times[-1/2,1/2]$, we get
\begin{align*}
\|v\|_{W^2_p(\Sigma')}&\leq N\big(\|v\|_{L_p(\Sigma)} + \|\mathcal{L}v\|_{L_p(\Sigma)} + \|\mathcal{B}v\|_{W^{1-1/p}_p(T)}\big)\\
&\leq N\big(\|u\|_{L_p(\Omega)} + \|(L-\lambda)u\|_{L_p(\Omega)} + \|Bu\|_{W^{1-1/p}_p(\p\Omega)}+\sqrt{\lambda}\|Bu\|_{L_p(\Omega)}\big).
\end{align*}
Notice that $D_{n+1,n+1}v=-\lambda v$, and we can find some $C>0$ independent of $\lambda$ such that
$$
C<\|\cos(\sqrt{\lambda}\cdot)\|_{L_p(-1/2,1/2)}<\|\cos(\sqrt{\lambda}\cdot)\|_{L_p(-1,1)}<2^{1/p} ,\quad\forall \lambda.
$$
Then,
$$\|v\|_{W^2_p(\Sigma')}\geq C(\|u\|_{W^2_p(\Omega)} + \lambda\|u\|_{L_p(\Omega)}).$$
Substituting back and choosing $\lambda$ large enough such that $C\lambda > N$, we get \eqref{removezeroorder}.
\end{proof}
\begin{corollary}\label{largelambdaexistenceCor}
Under the assumptions of Theorem \ref{maintheorem} with $\epsi_0$ being further smaller and $\lambda>\lambda_0$ as in Lemma \ref{removeuLem}, there exists a unique $W^{2}_p$ solution to \eqref{largelambda}.
\end{corollary}
\begin{proof}
Uniqueness is clear from the coercive estimate \eqref{removezeroorder}. We will focus on the existence.

First, noting that if $\p \Omega$ is smooth (say, $C^{1,1}$), the a priori estimate \eqref{removezeroorder} immediately gives us the solvability: one can first solve $$\begin{cases}
(\Delta - \lambda) u = f \quad &\mbox{ in } \Omega,\\
u + \frac{\p u}{\p n} = g \quad &\mbox{ on } \p \Omega,
\end{cases}$$
where $n$ is the outer normal direction, then use the method of continuity. Such argument and results can be found in \cite{MR2260015}.

For $\Omega$ with the small Lipschitz property, we approximate from the interior by $\{\Omega_k \in C^{1,1}\}_{k}\uparrow\Omega$. Moreover, we can require that all the $\Omega_k$ are $C\delta^2\epsi_0$-Lipschitz, where $C$ is a universal constant. Due to this and the continuity of $b$, we may further require that $b\cdot n_k\geq \abs{b}\delta/2$ for all $k$, where $n_k$ is the unit outer normal direction of $\Omega_k$. Now solve in $W^2_p(\Omega_k)$ for
\begin{equation*}
\begin{cases}
(L-\lambda)u_k = f \quad &\mbox{in } \Omega_k,\\
Bu_k = g\quad &\mbox{on } \partial \Omega_k.
\end{cases}
\end{equation*}
Here to make sense of the boundary condition, we need to extend $g$ which is only given on the boundary to $W^1_p(\Omega)$ as the operator $E$ defined in Section \ref{mainstepssection}. Notice that since the constant $N$ in \eqref{removezeroorder} depends on the regularity of $\Omega$ only through its Lipschitz bound and the radius in the small Lipschitz property, then $\{\|u_k\|_{W^2_p(\Omega_k)}\}_k$ are uniform bounded. We can use the following argument to get a subsequence $u_{k_i}\rightarrow u$ weakly in $W^2_p(\Omega)$.

For $\Omega_1$, noting that $\{\|u_k\|_{W^2_p(\Omega_1)}\}_k$ are uniformly bounded, we can find a subsequence $\{u_{k^1_j}\}_j$ which converges weakly in $W^2_p(\Omega_1)$. Similarly, we can find a further subsequence $\{u_{k^2_j}\}_j$ converges weakly in $W^2_p(\Omega_2)$. Repeating this process, we find $\{u_{k^i_j}\}_{i,j}$. Now a diagonal argument will give us the required converging weakly in $W^2_p(\Omega)$ subsequence. Denote this limit function by $u$.

Clearly $(L-\lambda)u = f$ in $\Omega$, we are left to check the boundary condition $Bu = g$. This is equivalent to say $Bu-g \in \mathring{W}^1_p(\Omega)$.
		
Since $Bu_k-g\in \mathring{W}^1_p(\Omega_k)$, we can take zero extension to $\mathring{W}^1_p(\Omega)$. Notice that $\mathring{W}^1_p(\Omega)$ is closed under the weak-$W^1_p$ topology, we infer that the limit $u$ has to satisfy $Bu-g \in \mathring{W}^1_p(\Omega)$.
\end{proof}

Now we are in the position of proving Theorem \ref{well-posednessThm}.
\begin{proof}[Proof of Theorem \ref{well-posednessThm}]
We aim to prove a uniform a priori estimate:
\begin{equation}\label{apriori}
\|u\|_{W^2_p(\Omega)}\leq N\big(\|(L-\lambda)u\|_{L_p(\Omega)} + \|g\|_{W^{1-1/p}_p(\p\Omega)}\big),
\end{equation}
where $\lambda\in [0,\lambda_0]$, $\lambda_0$ is the constant given in Lemma \ref{removeuLem}, and the constant $N$ is chosen to be independent of $\lambda$. Once we have this, the uniqueness and \eqref{nodependenceapriori} can be obtained by letting $\lambda=0$. For the existence, one only need to use the method of continuity and the large $\lambda$ existence result in Corollary \ref{largelambdaexistenceCor}.

Now we are left to prove \eqref{apriori}. Actually this can be further reduced to \eqref{nodependenceapriori}. First \eqref{nodependenceapriori} gives us \eqref{apriori} with $N=N(\lambda)$. Then we only need to find an upper bound of $N(\lambda)$, $\lambda \in [0,\lambda_0]$. This upper bound can be found using a compactness argument: for $\epsi$ sufficiently small,
\begin{align*}
\|u\|_{W^2_p(\Omega)} &\leq N(\lambda)\big(\|(L-\lambda)u\|_{L_p(\Omega)} + \|g\|_{W^{1-1/p}_p(\p\Omega)}\big)\\
&\leq N(\lambda)\big(\|(L-\lambda \pm\epsi)u\|_{L_p(\Omega)} + \|g\|_{W^{1-1/p}_p(\p\Omega)}\big) + N(\lambda)\epsi\|u\|_{L_p(\Omega)}
\end{align*}
implies
\begin{equation*}		
\|u\|_{W^2_p(\Omega)}\leq \frac{N(\lambda)}{1-N(\lambda)\epsi}\big(\|(L-\lambda \pm\epsi)u\|_{L_p(\Omega)} + \|g\|_{W^{1-1/p}_p(\p\Omega)}\big).
\end{equation*}
Then for every $\lambda\geq 0$, we can find a neighborhood $(\lambda-\epsi,\lambda+\epsi)$ on which \eqref{apriori} holds with a uniform constant $N$. Since $[0,\lambda_0]$ is compact, a finite upper bound of $N$ is attained.
	
Now we only need to prove \eqref{nodependenceapriori} under additional conditions \eqref{sign}-\eqref{alpha>d}. We first prove the following uniqueness result:
\begin{equation}\label{uniqueness}
\begin{cases}
u \in W^2_p(\Omega),\\
Lu = 0\quad &\mbox{ in } \Omega,\\
Bu = 0\quad &\mbox{ on } \p \Omega,
\end{cases}
\quad \mbox{has only zero solution}.
\end{equation}
When $p>d$, this uniqueness result is of Aleksandrov-Bakelman-Pucci type. Such result for the oblique derivative problem can be found in \cite[Corollary 2.5]{MR906819}: under the assumptions of Theorem \ref{well-posednessThm}, $u$ must be a constant. Then we obtain $u\equiv 0$ from \eqref{notalwayszero}. Note in this step, the boundedness of $\Omega$ is needed.

Now, for general $p\in (1,\infty)$, we need to use large $\lambda$ well-posedness from Lemma \ref{removeuLem} to improve regularity, i.e., to show that $u\in W^2_{d+\epsi}$. From Sobolev embedding and $u\in W^2_p$, we get $u \in L_q, p<q<\frac{dp}{d-2p}$. Now rewrite the equation as $(L-\lambda)u=-\lambda u$, and take $\lambda$ large enough. The $W^2_q$-existence and the $W^2_p$-uniqueness will tell us $u \in W^2_q$. Repeating if needed, we finally get $u\in W^2_{d+\epsi}$, where $\epsi$ satisfies $1-\frac{1}{d+\epsi}<\alpha$ and $\alpha$ is the H\"{o}lder exponent of the boundary data as in \eqref{alpha>d}. Then we can apply the result in \cite{MR906819} to get $u = 0$.
		
Passing from \eqref{uniqueness} to \eqref{nodependenceapriori} is a standard contradiction argument. Suppose \eqref{nodependenceapriori} were not true. With help of \eqref{mainestimate}, for all $k=1,2,\ldots$, there exist $u_k$ such that \begin{equation}\label{contradiction}
\|u_k\|_{L_p(\Omega)}> k (\|Lu_k\|_{L_p(\Omega)} + \|Bu_k\|_{W^{1-1/p}_p(\p \Omega)}).
\end{equation}
Without loss of generality, we take $\|u_k\|_{L_p(\Omega)}=1$. By \eqref{mainestimate} and \eqref{contradiction},
\begin{align*}
\|u_k\|_{W^2_p(\Omega)}&\leq N(\|Lu_k\|_{L_p(\Omega)}+\|Bu_k\|_{W^{1-1/p}_p(\p\Omega)} + \|u_k\|_{L_p(\Omega)})\\
&< N\big(\frac{1}{k}+1\big)\|u_k\|_{L_{p}}\\
&\leq 2N.
\end{align*}
Now, since $u_k$ is uniformly bounded in $W^2_p$, passing to a subsequence we have $u_{k_i}\rightarrow u$ weakly in $W^2_p(\Omega)$, and $u_{k_i}\rightarrow u$ strongly in $L_p(\Omega)$. Using \eqref{contradiction}, we obtain $\|Lu_{k_i}\|_{L_p}\rightarrow 0$, and $\|Bu_{k_i}\|_{W^{1-1/p}_p}\rightarrow 0$. We can deduce that $Lu=0, Bu=0$, and hence $u=0$ by the uniqueness. We have reached a contradiction, since $u_{k_i} \rightarrow u$ strongly in $L_p(\Omega)$ implies $\|u\|_{L_p(\Omega)}=1$.
\end{proof}

\section{Nonlinear equations}\label{nonlinearsection}
Similar to the Schauder estimate in \cite{Safonov}, our method also works for fully nonlinear equations with proper convexity conditions. In this section, we show this for Bellman equations which can be written as follows:
\begin{equation*}
\sup_{\omega}\{L^\omega u + f(\omega,x)\}:= \sup_{\omega}\{a_{ij}(\omega,x)D_{ij}u + a_i(\omega,x)D_i u + a_0(\omega,x) u + f(\omega,x)\} = 0.
\end{equation*}
Compared to the linear case, we have the following assumptions which are uniform in $\omega$: $a_{ij}(\omega,x)$ are measurable in $x$, symmetric, and satisfy
\begin{equation}\label{unifVMOciefficients}
\nu |\xi|^2 \leq a_{ij}(\omega,x)\xi_i\xi_j \leq \nu^{-1}|\xi|^2,\quad \forall x, \omega.
\end{equation}
\begin{equation}\label{unifboundeda}
\exists K>0, \text{ such that } \|a_i(\omega,\cdot), a_0(\omega,\cdot)\|_\infty<K,\quad \forall \omega.
\end{equation}

In contrast to Assumption \ref{smallBMOAss}, we state the following uniformly small BMO condition.
\begin{assumption}[$r_0,\theta$]\label{uniformVMOAss}
For a constant $\theta>0$, there exists an $r_0>0$ such that
\begin{equation*}
\sup_{x\in\Omega, 0<r< r_0}\dashint_{B_r(x)\cap\Omega}\sup_\omega|a(\omega,y)-(a)_{B_r(x)\cap\Omega}(\omega)| \,dy \leq \theta,
\end{equation*}
where $\theta$ is a positive constant to be specified later.
\end{assumption}
Under these settings, we have the following result which is analogous to Theorem \ref{maintheorem}.
\begin{theorem}\label{mainbellmanThm}
Assume that $\Omega$ is a Lipschitz domain and $p>d$. Let $u \in W^2_p(\Omega)$ be a solution to the Bellman equation:	
\begin{equation}\label{bellman}		
\begin{cases}
\sup_{\omega}\{L^\omega u + f(\omega,x)\}=0\quad &\mbox{in } \Omega,\\
Bu=b_0 u + b_iD_i u = g\quad &\mbox{on } \partial \Omega.
\end{cases}
\end{equation}
Assume that \eqref{bregularity},  \eqref{oblique}, \eqref{unifVMOciefficients} and \eqref{unifboundeda} hold. Also, $\bar{f}(x):= \sup_{\omega} |f(\omega,x)| \in L_p(\Omega), g\in W^{1-1/p}_p(\p \Omega)$. Then there exist $\theta_0=\theta_0(d,p,\nu)>0$ and $\epsi_0=\epsi_0(d,p,\nu, K)>0$ small enough, such that if the domain $\Omega$ is $\delta^2\epsi_0$-Lipschitz and Assumption \ref{uniformVMOAss} $(r_0, \delta\theta_0)$ holds, then we have,
\begin{equation}\label{bellmanmainestimte}
\|u\|_{W^2_p(\Omega)} \leq N\big(\|u\|_{L_p(\Omega)} + \|\bar{f}\|_{L_p(\Omega)} + \|g\|_{W^{1-1/p}_p(\partial\Omega)}\big).
\end{equation}
Here $N=N(d,p,\nu, K,\|b_0,b_0\|_\alpha, r_0, R_0,\delta)$ is a constant.
\end{theorem}
For the proof, we follow the scheme in Section \ref{mainstepssection} which is given for the linear case there. Recall for the linear case we prove the theorem in 3 steps: proving under the homogeneous boundary condition $\frac{\p u}{\p y^d}=0$; constructing an extension; the perturbation argument. The latter two steps still work since they only deal with $\p\Omega$ and the boundary operator $B$, and have nothing to do with the elliptic operator. Hence, we only need to give the proof of the first step.

Recall that in Section \ref{mainstepssection}, we use the regularized distance to flatten the boundary, then apply Hardy's inequality and the half space result. This argument still works except that we need to prove the corresponding $W^2_p$ estimate for the Bellman equation in half space with the Neumann boundary condition. In the following lemma, we adopt the assumptions in Theorem \ref{mainbellmanThm}, but in all places we replace $\Omega$ by $\mathbb{R}^d_{+}$.

\begin{lemma} \label{bellmanmodelLem}
Assume that $u\in W^2_p(\mathbb{R}^d_{+})$ solves
\begin{equation*}		
\begin{cases}
\sup_{\omega}\{L^\omega u + f(\omega,x)\}=0\quad &\mbox{in } \mathbb{R}^d_{+},\\
\frac{\p u}{\p x^d} = 0\quad &\mbox{on } \p\mathbb{R}^d_{+}.
\end{cases}
\end{equation*}
There exists a constant $\theta_0=\theta_0(d,p,\nu)>0$, such that if Assumption \ref{uniformVMOAss} is satisfied with $(r_0, \theta_0)$, then we have	
\begin{equation}\label{vmohalfspaceneumann}
\|D^2u\|_{L_p(\mathbb{R}^d_{+})} \leq N (\|\bar{f}\|_{L_p(\mathbb{R}^d_{+})} + \|u\|_{L_p(\mathbb{R}^d_{+})}).
\end{equation}
Again $\bar{f}(x):=\sup_{\omega}|f(\omega,x)|$. Here the constant $N$ depends on $d$, $p$, $\nu$, $K$, and $r_0$.
\end{lemma}

Before we start the proof, we would like to mention that there are similar results in \cite{bellmanvmointerior,bellmanvmodirichlet}. In \cite{bellmanvmointerior}, an interior $W^2_p$ estimate for Bellman equations with small BMO coefficients was established. The corresponding boundary estimate under the Dirichlet boundary condition was proved in \cite{bellmanvmodirichlet}. Our proof here follows similar steps in these two papers.

\begin{proof}
Using localization techniques, we may assume $u$ has compact support in $B^{+}_{r_0}(z)$. Here $r_0$ is the radius in Assumption \ref{uniformVMOAss} and $z \in \mathbb{R}^d_{+}$. For such $u$, we aim to prove
	
\begin{equation}\label{average}
\begin{split}	
&\dashint_{B_r^{+}(x_0)}\dashint_{B_r^{+}(x_0)}|D^2u(x)-D^2u(y)|^{\gamma}\,dx\,dy \leq
N\kappa^d\big(\dashint_{B_{\kappa r}^{+}(x_0)}(\bar{f}+|Du|+|u|)^d \,dx\big)^{\gamma/d}\\
&\quad +N\kappa^d\theta^{(1-1/\beta)\gamma/d}\big(\dashint_{B_{\kappa r}^{+}(x_0)}|D^2u|^{\beta d} \,dx\big)^{\gamma/(\beta d)} + N\kappa^{-\gamma\bar{\alpha}}\big(\dashint_{B_{\kappa r}^{+}(x_0)}|D^2u|^d\,dx\big)^{\gamma/d}.
\end{split}
\end{equation}
Here $r\in(0,\infty),\beta\in(1,\infty),\kappa\geq 16$ are parameters which can be chosen arbitrarily, and $\bar{\alpha}=\bar{\alpha}(d,\nu)\in(0,1),N=N(d,\nu),\gamma=\gamma(d,\nu)\in(0,1]$ are all determined constants. The center $x_0$ can be any fixed point in $\mathbb{R}^d_{+}$.
	
When $B_{\kappa r}(x_0) \subset \mathbb{R}^{d}_{+}$, this is the interior estimate in \cite{bellmanvmointerior}. Hence we only need to prove for $x_0$ close to the boundary. For simplicity, here we only prove for $x_0 \in \p\mathbb{R}^d_{+}$, and write $B^{+}_{\kappa r}$ for $B^{+}_{\kappa r}(x_0)$. For general $x_0$ with $B_{\kappa r}(x_0) \not\subset \mathbb{R}^{d}_{+}$, the proof is similar.

We shall use the frozen coefficient argument. Define $$\bar{a}_{ij}^{\omega}:=\begin{cases}
(a_{ij}^\omega)_{B^{+}_{\kappa r}}\quad &\mbox{ if } \kappa r\leq r_0,\\
(a_{ij}^\omega)_{B^{+}_{r_0}(z)}\quad &\mbox{ if } \kappa r> r_0.
\end{cases}$$
Then we decompose $u=v+w$, where $v$ solves the boundary value problem, i.e.,
\begin{equation}\label{constantbellman}
\begin{cases}
\sup\{\bar{a}_{ij}(\omega)D_{ij}v\}=0\quad &\mbox{in }B_{\kappa r}^{+},\\
\frac{\p v}{\p x^d}=0 \quad &\mbox{on } \Gamma:= \p\mathbb{R}^d_{+}\cap B_{\kappa r}^{+},\\
v=\hat{u}:=u-\sum_{i=2}^{d}x^i(D_iu)_{B_{\kappa r}^{+}} - (u)_{B_{\kappa r}^{+}} \quad &\mbox{on } \p B_{\kappa r}^{+}\setminus\Gamma.
\end{cases}
\end{equation}
Such $\hat{u}$ has properties: $D^2\hat{u} = D^2u$, and by even extension and Sobolev and Poincar\'e inequalities,
$$
\sup_{B_{\kappa r}^{+}}|\hat{u}|\leq \kappa r \|D^2u\|_{L_d(B_{\kappa r}^{+})}.
$$
See \cite[Lemma~2.1]{bellmanvmodirichlet}.
For the equation $\eqref{constantbellman}$, we can find a Krylov-Evans type existence result for mixed boundary conditions in \cite[Theorem~8.1]{safonov1994boundary}: there exists $v\in C^{2,\bar{\alpha}}_{loc}(B_{\kappa r}^{+}\cup \Gamma)\cap C^0(\overline{B_{\kappa r}^{+}})$ solving \eqref{constantbellman}, where $\bar{\alpha}=\bar{\alpha}(d,\nu)$ is some constant between $0$ and $1$. The strong maximum principle tells us: $$\sup_{B_{\kappa r}^{+}}|v|\leq\sup_{\p B_{\kappa r}^{+}\setminus \Gamma}|\hat{u}|.$$
From the equation of $v$, the Krylov-Evans estimate in \cite[Theorem~8.1]{safonov1994boundary} and a dilation argument, we have: for all $x,y \in B_{r}^{+}(x_0)$,
\begin{align*}
|D^2v(x)-D^2v(y)|&\leq |x-y|^{\bar{\alpha}} [D^2v]_{C^{\bar{\alpha}}(B^{+}_r(x_0))}\\
&\leq N|x-y|^{\bar{\alpha}}(\kappa r-r)^{-2-\bar{\alpha}}\sup_{B_{\kappa r}^{+}} |v|\\
&\leq N|x-y|^{\bar{\alpha}}(\kappa r)^{-2-\bar{\alpha}}\sup_{\p B_{\kappa r}^{+}\setminus \Gamma} |\hat{u}|\\
&\leq N|x-y|^{\bar{\alpha}}(\kappa r)^{-1-\bar{\alpha}}\|D^2 u\|_{L_d(B_{\kappa r}^{+})}.
\end{align*}
Integrating $x,y$ in $B_r^{+}$, we obtain
\begin{equation}\label{vestimate}
\dashint_{B_r^{+}(x_0)}\dashint_{B_r^{+}(x_0)}|D^2v(x)-D^2v(y)|^{\gamma}\,dx\,dy \leq N\kappa^{-\gamma\bar{\alpha}}\big(\dashint_{B_{\kappa r}^{+}}|D^2u|^d\big)^{\gamma/d}, \quad \forall \gamma \in (0,1].
\end{equation}
Now let us consider $D^2w$. Since $D^2\hat{u}=D^2u$, it is equivalent to discuss $\hat{w}:= \hat{u}-v$ instead of $w$  $(=u-v)$. From the equation of $w$, one can simply show that $\hat{w}$ satisfies
\begin{align*}
\sup\{\bar{a}_{ij}(\omega)D_{ij}\hat{w}+f(\omega,x) + a_i(\omega,x)D_iu + a_0(\omega,x)u + (a_{ij}(\omega,x)- \bar{a}_{ij}(\omega))D_{ij}u\} \geq 0,\\ \inf\{\bar{a}_{ij}(\omega)D_{ij}\hat{w}+f(\omega,x) + a_i(\omega,x)D_iu + a_0(\omega,x)u + (a_{ij}(\omega,x)- \bar{a}_{ij}(\omega))D_{ij}u\} \leq 0.
\end{align*}
One important observation here is that, we can rewrite the equation for $\hat{w}$ as follows,
\begin{equation}\label{equationforw}
\begin{cases}
\hat{L}\hat{w}:=\hat{a}_{ij}(x)D_{ij}\hat{w}(x) = \hat{f}(x)\quad &\mbox{in } B_{\kappa r}^{+},\\
\frac{\p\hat{w}}{\p x^d}=0\quad &\mbox{on }\Gamma,\\
\hat{w}=0\quad &\mbox{on }\p B_{\kappa r}^{+}\setminus\Gamma.
\end{cases}
\end{equation}
Here $\hat{a}_{ij}$ are measurable and symmetric with $\nu |\xi|^2 \leq \hat{a}_{ij}\xi_i\xi_j \leq \nu^{-1}|\xi|^2$, and $\hat{f}$ is also measurable, with $$|\hat{f}|\leq \bar{f} + K(|Du|+|u|) + \sup_\omega |a_{ij}(\omega,x)- \bar{a}_{ij}(\omega)||D^2u|.$$
From \cite{fanghualinroughcoefficient}, we have the following estimate for uniform elliptic operators with only measurable coefficients:
\begin{equation}\label{fhlin}
\|D^2h\|_{L_\gamma(B_1)}\leq N \|L_\nu h\|_{L_d(B_1)}.
\end{equation}
Here $h\in W^2_d(B_1), h=0$ on $\partial B_1$, $L_\nu$ is any bounded and uniform elliptic operator (i.e. \eqref{unifellipticciefficients} holds with constant $\nu$) of which the coefficients are only measurable, and $\gamma=\gamma(d,\nu)\in(0,1]$. For our equation \eqref{equationforw} which is in half ball with the Neumann boundary condition, to apply \eqref{fhlin}, we take even extension for $\hat{w},\hat{f}$, and correspondingly for $\hat{a}_{ij}$. We derive that for some $\gamma=\gamma(d,\nu) \in (0,1]$,
\begin{align*}
&\big(\dashint_{B^{+}_{r}(x_0)} |D^2 \hat{w}(x)|^\gamma \,dx\big)^{1/\gamma} \leq \kappa^{d/\gamma} \big(\dashint_{B_{\kappa r}^{+}}|D^2\hat{w}(x)|^\gamma \,dx\big)^{1/\gamma}\\
&\leq N\kappa^{d/\gamma}\big(\dashint_{B_{\kappa r}^{+}}|\hat{L}\hat{w}(x)|^d \,dx\big)^{1/d}\\
&\leq N\kappa^{d/\gamma}\bigg(\dashint_{B_{\kappa r}^{+}}\big(\bar{f} + K(|Du|+|u|) + \sup_\omega |a_{ij}(\omega,x)- \bar{a}_{ij}(\omega)||D^2u|\big)^d\bigg)^{1/d}\\
&\leq N\kappa^{d/\gamma}\big(\dashint_{B_{\kappa r}^{+}}(\bar{f} + |Du| + |u|)\big)^{1/d} + N \kappa^{d/\gamma}\theta^{(1-1/\beta)/d}\big(\dashint_{B_{\kappa r}^{+}} |D^2u|^{\beta d}\big)^{1/(\beta d)}.
\end{align*}
Here $\beta>1$ is any constant satisfying $\beta d<p$. In the last step, we use the fact that $u$ has compact support in $B^{+}_{r_0}(z)$, H\"{o}lder's inequality, and Assumption \ref{uniformVMOAss}:
\begin{align*}
\|\sup_\omega |a_{ij}(\omega,x)-& \bar{a}_{ij}(\omega)|D^2u\|_{L_{d}(B_{\kappa r}^{+})}\\
\leq& \|\sup_\omega |a_{ij}(\omega,x)- \bar{a}_{ij}(\omega)|1_{B^{+}_{r_0}(z)}\|_{L_{\beta d/(\beta-1)}(B_{\kappa r}^{+})}\|D^2u\|_{L_{\beta d}(B_{\kappa r}^{+})}\\
\leq& N(r_0,d)\theta^{(1-1/\beta)/d}\|D^2u\|_{L_{\beta d}(B_{\kappa r}^{+})}.
\end{align*}
Then we have
\begin{equation}\label{westimate}
\begin{split} &\dashint_{B_r^{+}(x_0)}\dashint_{B_r^{+}(x_0)}|D^2\hat{w}(x)-D^2\hat{w}(y)|^{\gamma}\,dx\,dy\\
&\quad \leq N\kappa^d\big(\dashint_{B_{\kappa r}^{+}(x_0)}(\bar{f}+|Du|+|u|)^d \,dx\big)^{\gamma/d} + N\kappa^d\theta^{(1-1/\beta)\gamma/d}\big(\dashint_{B_{\kappa r}^{+}}|D^2u|^{\beta d} \,dx\big)^{\gamma/(\beta d)}.
\end{split}
\end{equation}
Combining \eqref{vestimate} and \eqref{westimate}, and noticing that $D^2\hat{w}=D^2w$, we obtain \eqref{average}.
	
Once we have \eqref{average}, take supreme in $r$, we have for any $\kappa\geq 16$,
\begin{equation}\label{maximalsharp}
\begin{split}
(D^2u)^{\#}_\gamma\leq& N \kappa ^{d/\gamma}\mathbb{M}^{1/d}(\bar{f}^d) + N \kappa ^{d/\gamma}\mathbb{M}^{1/d}(|u|^d) + N \kappa ^{d/\gamma}\mathbb{M}^{1/d}(|Du|^d)\\
&+ N\kappa^{d/\gamma}\theta^{(1-1/\beta)/d}\mathbb{M}^{1/(\beta d)}(|D^2u|^{\beta d}) + N\kappa^{-\bar{\alpha}}\mathbb{M}^{1/d}(|D^2u|^d).
\end{split}
\end{equation}
Here $\mathbb{M}$ is the classical centered maximal function
$$\mathbb{M}f(x_0):=\sup_{r>0}\dashint_{B^{+}_r(x_0)}|f(x)|\,dx$$
and $(D^2u)^{\#}_\gamma$ is the following sharp function
$$(D^2u)^{\#}_\gamma(x_0):=	\sup_{r>0}\Big(\dashint_{B_r^{+}(x_0)}\dashint_{B_r^{+}(x_0)}
|D^2u(x)-D^2u(y)|^{\gamma}\,dx\,dy\Big)^{1/\gamma}.$$
Now, taking the $L_p$ norm for both sides of \eqref{maximalsharp}, applying the Hardy-Littlewood theorem and a Fefferman-Stein type sharp function theorem which can be found in the appendix of \cite{bellmanvmointerior}, we get for $p>\beta d$,
$$\|D^2u\|_p\leq N\kappa^{d/\gamma}(\|\bar{f}\|_p+\|Du\|_p+\|u\|_p) + N(\kappa^{d/\gamma}\theta^{(1-1/\beta)/d} + \kappa^{-\bar{\alpha}})\|D^2u\|_p,$$
where $\|\cdot\|_p$ is the abbreviation for $\|\cdot\|_{L_p(\mathbb{R}^d_{+})}$. We are left to choose $\kappa$ sufficient large and then $\theta$ sufficient small to absorb $\|D^2u\|_p$ into the left-hand side, and then use the interpolation inequality to obtain \eqref{vmohalfspaceneumann}.

Noting that in the above computation, we require $u$ to have compact support in $B^{+}_{r_0}(z)$. We may reduce our original problem to this, by using partition of unity $\xi_i$ for the covering $\{B^{+}_{r_0}(z_i)\}_i$ of half space. In \eqref{average}, we substitute $u$ by $u\xi_i$ and $\bar{f}$ by $\bar{f}\xi_i+\sup_\omega\{|[\xi_i,L^\omega]u|\}$, and then sum over all $i$. Here $[\cdot,\cdot]$ is the usual notation for commutators. Noting that all the extra terms are of lower order, this will only add $\|Du\|_p$ and $\|u\|_p$ to the right-hand side of our estimate. To get \eqref{vmohalfspaceneumann} we use the interpolation inequality $\|Du\|_p\leq \epsi \|D^2u\|_p + N(\epsi)\|u\|_p$.
\end{proof}
In a subsequent paper, we will consider the $W^2_p$ estimate and solvability for more general fully nonlinear operators with the oblique derivative boundary condition.

\appendix

\section{Regularized Distance for Small Lipschitz Domain}
In this section, we sketch the proof for Theorem \ref{RegularizedDistanceThm}. In \cite[Theorem~1.1]{LiebermanRegularizedDistance}, the construction of regularized distance in Lipschitz domain is given, along with property \eqref{distance}. As for properties \eqref{smalllip} and \eqref{secondorderderivative} concerning the derivatives, we modify the proof of \cite[Theorem~1.3]{LiebermanRegularizedDistance}, which requires $\p \Omega \in C^1$.
\begin{proof}[Proof of Theorem \ref{RegularizedDistanceThm}]
Let $g(y)=\psi(y')-y^d$. Since $\Omega$ is Lipschitz, as in \cite[Section~2]{LiebermanRegularizedDistance}, if we take $M=2(4/\delta^2+1)^{1/2}$, then
$$
[g]_1\leq M/2,\quad (M/2)^{-1}\leq g/d_y \leq M/2\quad \text{in}\, Q_{4R}(x_0)\setminus \Gamma_{4R}.
$$
Here $d_y$ is the distance between $y$ and the boundary.

We construct the regularized distance as the fixed point of the following:
\begin{equation}\label{fixedpoint}
\rho(y)=G(y,\rho(y)),
\end{equation}
where $G$ is the mollification of $g$, $$G(y,\tau):=\int_{|w|<1}g(y-\frac{\tau}{M}w)\phi(w)\,dw.$$
Here $\phi \in C_c^\infty(B_1), \phi\geq0, \int \phi =1$.

Clearly $[G(\cdot,\tau)]_1\leq [g]_1$. The key fact of such $G$ is that: $[G(y,\cdot)]_1\leq 1/2$. From the contraction mapping theorem, we get the unique fixed point in $C^0$, denoted by $\rho_0$. We first show $\rho_0\in C^{0,1}(Q_R)\cap C^\infty(Q_R\setminus \Gamma_R)$ and \eqref{distance}.

Using $[G(y,\cdot)]_1\leq 1/2$, from \begin{align*}
|\rho_0(y)-\rho_0(z)|&\leq|G(y,\rho_0(y))-G(y,\rho_0(z))|+|G(y,\rho_0(z))-G(z,\rho_0(z)|)\\
&\leq [G(y,\cdot)]_1|\rho_0(y)-\rho_0(z)|+[G(\cdot,\rho(z))]_1,
\end{align*}
we can obtain $$[\rho_0]_1\leq 2[G(\cdot,t)]_1\leq 2[g]_1 \leq M.$$ Hence $\rho_0 \in C^{0,1}(Q_R)$. The fact that $\rho_0 \in C^\infty(Q_R\setminus \Gamma_R)$ follows from implicit function theorem and the fact that $\phi \in C^\infty$.

For \eqref{distance}, we write $$|\rho_0(y)-G(y,0)|=|G(y,\rho_0(y))-G(y,0)|\leq \frac{1}{2}|\rho_0(y)|.$$ Noting that $G(y,0)=g(y)$, we have $\frac{2}{3}\leq \frac{\rho_0}{g(y)}\leq 2$. Since $g$ is a distance function, \eqref{distance} is proved.

Now we turn to proving \eqref{smalllip} and \eqref{secondorderderivative}. We modify the proof of \cite[Theorem~1.3]{LiebermanRegularizedDistance}, which works for $C^1$ domains, to adapt to the case of small Lipschitz domains. Denote $$G_i(y,\tau):= \frac{\p}{\p y^i}G(y,\tau),\quad i=1,\ldots,d, \quad  G_{d+1}(y,\tau):=\frac{\p}{\p \tau}G(y,\tau).$$
Differentiating both sides of \eqref{fixedpoint} at the fixed point $\rho_0$, and applying the chain rule, we get
\begin{align*}
|D_i\rho_0(y)-D_i\rho_0(z)|\leq |&G_i(y,\rho_0(y))-G_i(z,\rho_0(z))|+|D_i\rho_0(y)||G_{d+1}(y,\rho_0(y))-\\
&-G_{d+1}(z,\rho_0(z))|+||G_{d+1}(z,\rho_0(z))|D_i\rho_0(y)-D_i\rho_0(z)|.
\end{align*}
Again, using $[G(z,\cdot)]_1\leq 1/2, [\rho_0]_1\leq M$, we have
\begin{align}\label{Drhodifference}
&|D_i\rho_0(y)-D_i\rho_0(z)|\nonumber\\
&\leq 2|G_i(y,\rho_0(y))-G_i(z,\rho_0(z))| + 2M|G_{d+1}(y,\rho_0(y))-G_{d+1}(z,\rho_0(z))|.
\end{align}
Now we use the explicit expression for $g$, i.e., $g=\psi(y')-y^d$, and the small Lipschitz property \eqref{epsilonlipschitz} given in terms of $D\psi$. Noting that for $g$ Lipschitz continuous, we can still differentiate under the integral sign in \eqref{fixedpoint}.
\begin{equation*}
G_i(y,\rho_0(y))=\begin{cases}
\int_{|w|<1} -\phi(w)\,dw,\quad i=d,\\
\int_{|w|<1} D_i\psi(y'-\frac{\rho_0(y)}{M}w')\phi(w)\,dw,\quad i<d,
\end{cases}
\end{equation*}
\begin{equation*}
G_{d+1}(y,\rho_0(y))=\int_{|w|<1} -D_k\psi\big(y'-\frac{\rho_0(y)}{M}w'\big)\frac{w^k}{M}\phi(w)\,dw.
\end{equation*}
Substituting back to \eqref{Drhodifference} and using \eqref{epsilonlipschitz}:  $|D\psi(y')-D\psi(z')|<3\epsi_0$, we obtain
$$|D_i\rho(y)-D_i\rho(z)|\leq 2\cdot3\epsi_0 + 2M\cdot\frac{3\epsi_0}{M}=12\epsi_0.$$
The proof for \eqref{secondorderderivative} is similar.	Differentiating both sides of $\rho_0(y)=G(y,\rho_0(y))$, we get
$$D_{ij}\rho_0(y)=G_{ij} + G_{i,d+1}D_j\rho_0 + G_{j,d+1}D_i\rho_0 + G_{d+1,d+1}D_i\rho_0D_j\rho_0 + G_{d+1}D_{ij}\rho_0.$$
Here $G_{ij}(y,\tau) = \frac{\p^2}{\p y^i\p y^j}G(y,\tau)$, and so on. Using $[G(z,\cdot)]_1\leq 1/2, [\rho_0]_1\leq M$, we have
$$|D_{ij}\rho_0(y)| \leq 2(|G_{ij}| + M|G_{i,d+1}| + M|G_{j,d+1}| + M^2|G_{d+1,d+1}|).$$
Thus we only need to estimate $D^2G$. Here we will only compute $G_{ij}$ with $i,j<d$, while the others are similar.
\begin{align}
G_{ij}(y,\tau) &= D_i\int_{|w|<1}D_j\psi\big(y'-\frac{\tau}{M}w'\big)\phi(w)\,dw\nonumber\\
&=D_i\int_{|w|<1}\frac{M}{\tau}\psi\big(y'-\frac{\tau}{M}w'\big)D_j\phi(w)\,dw\label{ibp}\\
&=\int_{|w|<1}\frac{M}{\tau}D_i\psi\big(y'-\frac{\tau}{M}w'\big)D_j\phi(w)\,dw\label{beforedifference}.
\end{align}
Here in \eqref{ibp} we integrated by parts, noticing that
$$\frac{\p\psi(y'-\frac{\tau}{M}w')}{\p w^i}=-\frac{\tau}{M}D_i\psi\big(y'-\frac{\tau}{M}w'\big).$$
Since $\int_{|w|<1} D_j\phi(w)\,dw = 0$, we may rewrite \eqref{beforedifference} as follows:
$$
\int_{|w|<1}\frac{M}{\tau}
\big(D_i\psi(y'-\frac{\tau}{M}w')-D_i\psi(y')\big)D_j\phi(w)\,dw.
$$
Now we can apply \eqref{epsilonlipschitz} to deduce
\begin{align*}
|G_{ij}(y,\rho_0(y))|&\leq \int_{|w|<1}\frac{M}{|\rho_0(y)|}
\big|D_i\psi(y'-\frac{\rho_0(y)}{M}w')-D_i\psi(y')\big||D_j\phi(w)|\,dw\\
&\leq N\frac{\epsi_0}{|\rho_0(y)|}.
\end{align*}
The theorem is proved.
\end{proof}

\section{Inequalities for Regularized Mollification}
We give the proof of Lemma \ref{YoungLem}:
\begin{proof}
The proof of \eqref{young_Lp} follows from the Minkowski inequality,
\begin{equation*}
\big\|\int f(x,\cdot) \,dx\big\|_p\leq \int \|f(x,\cdot)\|_p \,dx.
\end{equation*}
From the definition of $\widetilde{g}$, we obtain
\begin{align}
\|\widetilde{g}(\cdot)\|_{L_{p}(\Omega_R)} &= \big\|\int g(\cdot-\frac{\rho_0(\cdot)}{M_1}w)\phi(w)\,dw\big\|_{L_p(\Omega_R)} \nonumber\\
&\leq \int\| g(\cdot-\frac{\rho_0(\cdot)}{M_1}w)\|_{L_p(\Omega_R)}\phi(w)\,dw \nonumber\\
&\leq 2^{1/p}\|g\|_{L_p(\Omega_{2R})}\|\phi\|_{L_1}\label{laststepyoung}.
\end{align}
Here \eqref{laststepyoung} is due to $y-\frac{\rho_0(y)}{M_1}w \in \Omega_{2R}, \forall y\in \Omega_R, w \in B_1$, and the Jacobian of the change of variables $y\mapsto (y-\frac{\rho_0(y)}{M_1}w)$ has a positive lower bound. One could compute this Jacobian as follows,
\begin{equation*}
\frac{\p (y-\frac{\rho_0(y)}{M_1}w)}{\p y} = I_d - \frac{D\rho_0}{M_1}w.
\end{equation*}
Thus we have $|I_d - \frac{D\rho_0(y)}{M_1}w|\geq \frac{1}{2}$ from our choice of $M_1$ in Definition \ref{DefMollification}.

The proof of \eqref{young_W1p} is similar, since
\begin{equation*}
D\widetilde{g}(y) = \int \big(I_d - \frac{D\rho_0}{M_1}w\big) Dg(y-\frac{\rho_0(y)}{M_1}w)\phi(w)\,dw
\end{equation*}
and $D\rho_0$ is bounded.
\end{proof}



\end{document}